\documentclass[12pt]{amsart}
\usepackage{amsmath}
\usepackage{geometry,amsfonts,amssymb,amsthm,txfonts,pxfonts,amscd,algorithmic,algorithm} 
\usepackage{refcheck}
\norefnames
\nocitenames
\def\struckint{\mathop{%
\def\mathpalette##1##2{\mathchoice{##1\displaystyle##2}%
  {##1\textstyle##2}{##1\scriptstyle##2}{##1\scriptscriptstyle##2}}%
\mathpalette
{\vbox\bgroup\baselineskip0pt\lineskiplimit-1000pt\lineskip-1000pt
\halign\bgroup\hfill$}
{##$\hfill\cr{\intop}\cr\diagup\cr\egroup\egroup}%
}\limits}
\newtheorem{theorem}{Theorem}[section]
\newtheorem{lemma}[theorem]{Lemma}
\newtheorem{corollary}[theorem]{Corollary}

\theoremstyle{remark}

\newtheorem{remark}[theorem]{Remark}

\newtheorem{example}[theorem]{Example}

\newtheorem{question}[theorem]{Question}

\DeclareMathOperator{\sech}{sech}
\DeclareMathOperator{\cotanh}{cotanh}

\DeclareMathOperator{\arccosh}{arccosh}
\DeclareMathOperator{\arcsinh}{arcsinh}

\DeclareMathOperator{\Beta}{B}

\begin{document}

\title{On Basmajian's identities, and other stories}

\author{Igor Rivin}
\address{School of Mathematics, University of St Andrews, Scotland}
\curraddr{Department of Mathematics, Temple University, Philadelphia}
\email{rivin@temple.edu}
\thanks{The author would like to thank Maryam Mirzakhani for helpful conversations, Susanna Dann and Vlad Yaskin for their interest, and Anton Petrunin for a nice observation.}
\date{\today} 


\begin{abstract}
We give a different perspective on the (by now) classic Basmajian identity, and point out some related results, both in the setting of hyperbolic manifolds, and in the polyhedral setting \emph{without} any group acting. We use some of our results to give some combinatorial consequences of hyperbolicity.
\end{abstract}
\thanks{This work began when the author was visiting at Stanford, and the results of Basmajian were brought to his attention by M. Mirzakhani -- he would like to thank her for interesting conversations. A previous version of this work was presented at an AIM conference in Convex Geometry -- the author is grateful for the feedback. The final version could not have been produced without the support (financial and logistic) of Brown University and ICERM, to both of whom the author owes a great deal.}
\maketitle
\section{Introduction} In his wonderful (and insufficiently well-known, for most of the intervening period) paper \cite{basmajianortho}, Ara Basmajian used a simple geometric idea to find relationships between the \emph{area} of a totally geodesic (or horospherical, or spherical) hypersurface $S$ in a hyperbolic manifold $M^n,$ and the sum of a certain function of lengths of the geodesics in the relative homotopy classes of paths between the boundary of $M^n$ and $S.$ Basmajian's work was rediscovered by Martin Bridgeman, which resulted in a number of interesting results (see, for example, \cite{bridgekahn,calespots}).

In this paper we give a simple approach to related identities, and use them to gain some insight into the geometry of hyperbolic manifolds, polygons, and polyhedra.

In Section \ref{volsec} we give some generalities on volumes.

In Section \ref{basicobj} we describe the basic geometric idea.

In Sections \ref{polytopes} and \ref{skeleta} we give some applications to the geometry of polygons, polyhedra, and higher-dimensional polytopes.

In Section \ref{prad} we show that polytopes with not too many faces (or not too many vertices) cannot contain large balls.

In Section \ref{kleinmod} we perform some computations in the Klein model of $\mathbb{H}^n,$ while in the Section \ref{integration} we do a couple of integrals.

In Section \ref{chimneys} we give the proof of an identity for manifolds with boundary.

In Section \ref{notchimneys} we give an application to the geometry of manifolds \emph{without} boundary.

\section{Volumes -- notation and formulae}
\label{volsec}
We will denote the area of the $k$-dimensional unit sphere in $\mathbb{E}^n$ by $\omega_{n-1}.$ Recall that 
\[
\omega_{k} = \dfrac{2 \pi^{(k+1)/2}}{\Gamma((k+1)/2)}.
\]
We denote the volume of the ball of radius $r$ in the sphere of dimension $k$ by $V^{\mathbb{S}}_k(r).$
We have 
\[
V^{\mathbb{S}}_n(r) = \omega_{n-1} \int_0^r \sin^{n-1}(t) d t.
\]

Note that we can get rid of the integral sign, as follows:
\[
\int_0^r \sin^{k-1}(t) dt = \int_0^{\sin(r) }\frac{s^{k-1}}{\sqrt{1-s^2}} d s = \frac12 \int_0^{\sin^2 r} u^{(k-2)/2}(1-u)^{-\frac12} d u = \Beta(\sin^2 r; \frac{k}2, \frac12),
\]
where $\Beta$ denotes the incomplete beta function.
The $\Beta$ function, can, in turn, be written as 
\[
\Beta(x; a, b) = \Beta(a, b) \left(
\frac{\Beta(x,;\{a\} + 1, b)} {\Beta(\{a\}+1, b)}+
\sum_{j=1}^{\lfloor a-1\rfloor} 
\frac{x^{a-j}(1-x)^b}{(a-j)\Beta(a-j, b)}\right)
\]
where $\Beta(a, b)$ is the usual beta function: 
\[\Beta(a, b) = \frac{\Gamma(a) \Gamma(b)}{\Gamma(a+b)}.\] Note that we stop the summation before $a$ becomes $0.$ Our $a$ is always half-integral, and $b=\frac12,$ so we need to know that 
\[
\begin{aligned}
\Beta(x; 1, \frac12) &= \int_0^x \frac{1}{\sqrt{1-x}} d x = 1- 2 \sqrt{1-x} \\
\Beta(1, 1/2) &= 2\\
\Beta(x; \frac32, \frac12) &= \int_0^x\frac{x^{\frac12}}{\sqrt{1-x}} d x = \arcsin(\sqrt{x}) - \sqrt{x(1-x)}\\
\Beta(\frac32, \frac12) &= \frac{\pi}2.
\end{aligned}
\]
We denote the volume of the ball of radius $r$ in hyperbolic space $\mathbb{H}^n$ of dimension $n$ by $V^{\mathbb{H}}_n(r).$
We have 
\[
V^{\mathbb{H}}_n(r) = \omega_{n-1} \int_0^r \sinh^{n-1}(t) dt.
\]
The areas of the spheres in the two geometries are given by 
\[
S_{n}^{\mathbb{H}}(r) = \omega_{n}\sinh^n r.
\]
and 
\[
S_n^{\mathbb{S}}(r) = \omega_n \sin^n r.
\]
\section{The basic observation} 
\label{basicobj}
Consider a collection of disjoint halfspaces $\Pi_1^+, \Pi_2^+, \dotsc, \Pi_n^+, \dotsc$ in $\mathbb{H}^n,$ such that their shadows cover the sphere at infinity (that is, the set covered has full measure). That is the same as saying that the shadows cover the sphere or radius $z$ centered at any point, or the visual sphere at any point. Let us use the last interpretation for simplicity. If $\Pi$ is a plane, then denote the visual angle subtended by $\Pi$ at $x$ by $\alpha_x(\Pi).$ If $V_{n-1}(r)$ is the volume of a spherical ball of radius $r.$ opur assumption implies that \begin{equation}
\label{visphere}
\boxed{
\omega_{n-1} = \sum_{i=1}^\infty V_{n-1} (\alpha_x(\Pi_i)),
}
\end{equation}
where $\omega_k$ is the surface area of the $k$-dimensional unit sphere.

Since Example \ref{alphaeg} tells us what $\alpha_x(\Pi_i)$ is in terms of the distance between $x$ and $\Pi_i,$ we can reqrite Eq. \eqref{visphere} as 
\begin{equation}
\label{asineq}
\boxed{
\omega_{n-1} = \sum_{i=1}^\infty V_{n-1} (\arcsin \sech d(x, \Pi_i)) =  \omega_{n-2} \sum_{i=1}^\infty \Beta\left(\sech^2(d(x,\Pi_i)).\frac{n-1}2, \frac12\right).
}
\end{equation}
If we assume that the planes $\Pi_i$ are disjoint, but do not cover the sphere at infinity, the identities \eqref{visphere} and \eqref{asineq} become inequality with the left hand side at least as big as the right hand side, while if we assume that the planes cover the sphere at infinity but are not necessarily disjoint, we get inequalities with the left hand side at most as big as the right hand side.
\begin{example}
\label{polygon}
Let $P$ be a convex \emph{ideal} polygon (with a possibly infinite number of sides $s_1, \dotsc, s_k, \dotsc$) in $\mathbb{H}^2,$ and let $x$ be an arbitrary point in the interior of $P.$ Let $d_i(x)$ be the distance of $x$ from $s_i.$ Then
\begin{equation}
\label{idealpolyid}
\pi = \sum_{i=1}^\infty \arcsin \sech d_i(x).
\end{equation}
Note that this is equivalent to the statement that the \emph{area} of an ideal polygon with $n$ vertices equals $(n-2)\pi.$
If $P$ is as above, but has some ideal and some finite vertex, then
\begin{equation}
\label{compactineq}
\pi \geq \sum_{i=1}^\infty \arcsin \sech d_i(x),
\end{equation}
 with equality if and only if $P$ is ideal. If $P$ is convex with some ideal and some \emph{hyperideal} vertices, then 
\begin{equation}
\label{hyperineq}
\pi \leq \sum_{i=1}^\infty \arcsin \sech d_i(x).
\end{equation}
\end{example}
\begin{question}
\label{idealpolygon} Does a collection of numbers satisfying the identity \eqref{idealpolyid} necessarily arise as the distances from some point in an interior of a convex ideal polygon to the sides? The same question for inequalities \eqref{compactineq} and \eqref{hyperineq}.
\end{question}
\section{Polytopes}
\label{polytopes}
Let $P$ be a convex polytope of dimension at least three. It is \emph{never} true that the planes of the top-dimensional faces of $P$ are disjoint. Let us now assume that $P$ has no hyperideal vertex, which means that for every point $x \in P,$ every point in the visual sphere of $x$ is covered \emph{at least} once by the projections of the faces of $P$ onto the sphere at infinity. This means that for any such point $x,$ 
\begin{equation}
\label{polyineq}
\omega_{n-1} < \sum_{F(P)} V_{n-1}(\arcsin \sech d(x, F)) = \omega_{n-2} \sum_{F(P}\Beta\left(\sech^2(d(x, F)).\frac{n-1}2, \frac12\right)
\end{equation}
Let us specialize to the case where $P$ is  a convex polyhedron in $\mathbb{H}^3.$ In this case, $\omega_2 = 4 \pi,$ and $V_2(t) = 2\pi  (1-\cos(t)).$
Since $\cos\arcsin x = \sqrt{1-x^2},$ and $\sqrt{1-\sech^2 t}  = \tanh t,$ we have 
\[
V_2(\arcsin \sech d(x, F)) = 2\pi(1-\tanh d(x, F)) = 4 \pi \frac{e^{-d(x, F)}}{e^{d(x, F)} + e^{-d(x, F)}} = 
\frac{4\pi}{e^{2 d(x, F)}+1}.
\]
Substituting into Eq. \ref{polyineq} we get (in three dimensions) the simpler to digest form
\begin{equation}
\label{sechform}
1 < \sum_{F(P)} \frac{1}{\exp(2 d(x, F)) + 1}.
\end{equation}
It should be noted that the right hand side of Eq. \eqref{sechform} has an uncanny resemblance to the form of McShane's identity (see \cite{mcshaneid,rioneint}).
\subsection{Acute Angled Polytopes}
Consider a polytope with all dihedral angles not exceeding $\pi/2.$ These are of particular interest as fundamental domains of groups generated by reflections (see \cite{vinbergsh,vinbergref,andreev1,andreev2,potyavinberg}). 
It is well-known (and easy to see) that the link of any vertex is combinatorially a simplex (in other words, such a polytope $P$ is \emph{simple}).
The following less trivial (but still not very hard) fact was shown by E.~M.~Andreev in \cite{andreevsharp}:
\begin{theorem}
\label{andreevthm}
Let $P$ be a polytope in $\mathbb{H}^n$ with all dihedral angles not exceeding $\pi/2.$ Then two facets of $P$ which are not incident in $P$ do not intersect.
\end{theorem}
The following is a simple corollary of the previous two observations:
\begin{theorem}
\label{atmostthm}
For an acute angled \emph{compact} polytope $P\in \mathbb{H}^n$ and a point $x \in P,$ each point on the visual sphere is covered \emph{at most} $n$ times by the shadows of the faces of $P.$
\end{theorem}
From Theorem \ref{atmostthm} it follows that we can add another side to inequality \ref{polyineq} for acute-angled polytopes:
\begin{corollary}
\label{acuteineqs}
If $P$ is an acute-angled \emph{compact} polytope in $\mathbb{H}^n,$
\[
\omega_{n-1} < \sum_{F(P)} V_{n-1}(\arcsin \sech d(x, F)) \leq n \omega_{n-1}.
\]
\end{corollary}

For a sanity check, let us see what happens when the polytope $P$ is very small, so close to Euclidean. In this case $d(x, F)$ is close to zero for all $F,$ so $\sech d(x, F)$ is close to $1,$ so $\arcsin \sech d(x, F)$ is close to $\pi/2.$ Which means that $V_{n-1}(\arcsin \sech d(x, F)) $ is close to $\omega_{n-1}/2$ (since a ball of radius $\pi/2$ is a hemisphere), so we get for a Euclidean polytope:
\begin{equation}
\label{eucineqs}
2 < \mathcal{F}(P)\leq 2n.
\end{equation}
where $\mathcal{F}(P)$ is the number of faces of $P.$
The left hand inequality is sharp (a polytope can have no fewer than $3$ faces), but more interestingly, so is the right hand inequality: A well-known theorem of H.~S.~M.~Coxeter \cite{coxref} states that the only polytopes in Euclidean $n$-dimensional space $\mathbb{E}^n$ are simplices and products of lower-dimensional simplices. It is easy to see that the polytope with the greatest number of faces is the $n$-cube, which has $2n$ faces, so the left-hand inequality of Eq. \eqref{eucineqs} is sharp.
\subsection{Circle packings and right angled ideal polyhedra}
In the case where our ``polytope'' is a collection of \emph{disjoint} hyperplanes $F_1, \dotsc, F_k,\dotsc$ the inequality \eqref{polyineq} is reversed, and we have 
\begin{equation}
\label{otherineq}
\omega_{n-1} \geq \sum_{F(P)} V_{n-1}(\arcsin \sech d(x, F)),
\end{equation}
where the inequality is strict whenever the number of hyperplanes is finite. This is true, in particular, when we have a \emph{circle packing} (where the hyperplanes are either disjoint or externally tangent). In three dimensions, it is known that there is a two-to-one correspondence (first noted by W. P. Thurston) between circle packings and ideal polyhedra all of whose dihedral angles are $\pi/2.$ Namely, the faces of such a polyhedra fall into two subsets, each of which constitutes a circle packing. From this observation and the inequality \eqref{otherineq} we immediately get the following result:
\begin{theorem}
Let $P$ be a right-angled ideal polyhedron, and $x$ a point in its interior. Then:
\begin{equation}
\frac12 \sum_{F(P)} V_{2}(\arcsin \sech d(x, F)) <\omega_{2} < \sum_{F(P)} V_{2}(\arcsin \sech d(x, F)).
\end{equation}
\end{theorem}
\begin{proof} 
The only relevant observation to make is that every point of the interior lies in \emph{at most} two  hyperplane shadows.
\end{proof}

\section{Higher Skeleta}
\label{skeleta}
Given a convex polytope $P,$ instead of looking at the visual measures of the top-dimensional facets of $P,$ we can, instead, look at the $k$-dimensional facets for \emph{any} $0\leq k < n,$ where $n$ is the dimension of the ambient space. The image of the $k$-skeleton $P_k$  of $P$ on the visual sphere at a point $x$ is the $k$-dimensional skeleton $S_k$ of the decomposition of the visual sphere at $x$ induced by $k.$ We then have the obvious inequality
\begin{equation}
\label{kskel}
\mu_k(S_k) \leq \sum_{F \in P_k} V_k (\arcsin \sech d(x, F)),
\end{equation}
with equality if the polytope is ideal, and $k< n-1. $ If the polytope is \emph{hyperideal} the inequalities (for $k<n-1$) go in the opposite direction, as before.
In the sequel, we will need a lower bound for $\mu_k(S_k).$ Since the projection of $P$ onto the visual sphere is a cell decomposition with all cells convex, it is quite clear that $\mu_k(S_k)\geq \omega_k.$ As pointed out by A.~Pertrun, this can be improved using the Crofton formula to 
\begin{equation}
\label{pest}
\mu_k(S_k) \geq \frac12 (n-k+2) \omega_k.
\end{equation}
Combining the estimates \eqref{kskel} and \eqref{pest}, we get the estimate
\begin{equation}
\label{skelest}
\frac12 (n-k+2) \omega_k \leq  \sum_{F \in P_k} V_k (\arcsin \sech d(x, F)) = \omega_{k-1} \sum_{F \in P_k} \Beta(\sech^2 d(x, F), \frac{k}2, \frac12).
\end{equation}
\section{Radius of polytopes}
\label{prad}
We define the \emph{radius} $r(P)$ of a complex polytope $P$ to be the radius of the largest ball contained in $P.$ We can now ask: 
\begin{question}
Can we estimate the radius of a polytope in $P\subset \mathbb{H}^n$ in terms of the number of faces, or the number of vertices of $P.$
\end{question}
Note, firstly, that the question is non-sensical in Euclidean space, so there is obviously no \emph{lower} bound on $r(P)$ in terms of combinatorial data (since very small polytopes are essentially Euclidean).

Note, secondly, that the fact that the radius of triangles in $\mathbb{H}^2$ is bounded above is, effectively, the reason why the hyperbolic plane is Gromov-hyperbolic. So, in a way, the results below give a higher-dimensional, higher-complexity way of  quantifying Gromov hyperbolicity.

Given the previous remark, it is not surprising that there \emph{is} an upper bound, and here it is:
\begin{theorem}
\label{ubound}
Given a polytope $P \subset \mathbb{H}^n,$ with $N$ top dimensional faces, 
$r(P) = O(\log N).$
\end{theorem}
\label{fbound}
\begin{proof} We apply the inequality \eqref{polyineq} together with the observation that $$\Beta(\sech^2(x), m/2, 1/2) \sim \exp(-m x),$$ for large $x.$
\end{proof}
Theorem \ref{fbound} is unsatisfying when we only have a bound on the number of \emph{vertices} of $P$ (since the number of faces may be exponential in the number of vertices), but in fact, the following is equally easy:
\begin{theorem}
\label{vbound2}
Given a polytope $P \subset \mathbb{H}^n,$ with $N$ vertices 
$r(P) = O(\log N).$
\end{theorem}
\begin{proof}
We use the estimate \eqref{skelest} for the $1$-skeleton. Since the number of edges is at most quadratic in the number of vertices, the result follows (in fact, with the same implied constant as in Theorem \ref{fbound}.
\end{proof}
A corollary of Theorems \ref{fbound} and \ref{vbound2} is the following combinatorial analogue of the exponential growth of volume in hyperbolic spaces:
\begin{corollary}
\label{spherecor}
Suppose $P$ is  a polytope in $\mathbb{H}^n,$ such that the Hausdorff distance between $P$ and a sphere $S(x_0, r)$ of radius $r$ around $x_0$ is bounded by $c r$ for some constant $c < 1.$ Then, the number of vertices and the number of faces of $P$ are both exponential in $r.$
\end{corollary}s

\section{The Klein Model}
\label{kleinmod}
The following formula is well-known and can be found in \cite{riasymp}: Let $p, q$ be two points in the unit disk $B_0(1),$ thought of as the Klein model of hyperbolic space.  Then the hyperbolic distance between $p$ and $q$ is given by:
\begin{equation}
\label{kleindist}
d(p, q) = \arccosh\left(\dfrac{1-p \cdot q}{\sqrt{1-p \cdot p} \sqrt{1-q\cdot q}}\right).
\end{equation}
\begin{example}
\label{origdist}
If $p$ is at the origin ($p=\mathbf{0}$), then 
\begin{equation}
\label{kzerodist}
d(\mathbf{0}, q) = \arccosh(1/\sqrt{1-q \cdot q}).
\end{equation}.
\end{example}
We would now like to find the distance between a point $q$ and a hyperplane $\Pi$ in the Klein model. We can specify the hyperplane $\Pi$ by its \emph{polar}. That is, the intersection of $\Pi$ with the unit sphere $\partial B_0(1)$ is a codimension one sphere $S_\Pi,$ and there exists a unique point $p = \Pi^*,$ such that the base of the tangent cone centered at $p$ to the unit sphere is precisely $S_\Pi.$  Furthermore, the distance from $q$ to $\Pi$ is realized by the segment between $q$ and $\Pi$ of the straight line containing $p$ and $q.$ Some manipulation of the formula \eqref{kleindist} gives us the following formula between $q$ and $\Pi = p^*.$
\begin{equation}
\label{pointplanedist}
d(q, p^*) = \arcsinh\left(\dfrac{1- p \cdot q}{\sqrt{p \cdot p -1} \sqrt{1- q \cdot q}.}\right).
\end{equation}
Note that the distance is the \emph{signed} distance, and is negative if $q$ is between $\Pi$ and $p.$
\begin{example}
\label{alphaeg}
Suppose $q=\mathbf{0}.$ Then
\[
d(\mathbf{0}, p^*) = \arcsinh(1/(\sqrt{ p \cdot  p - 1}).
\]
\end{example}
\begin{example}
Suppose $p^*$ is an equatorial plane, in which case $p$ is a point at infinity, so the formula \eqref{pointplanedist} does not apply directly. However, if $p$ is in the direction $v,$ with $v$ a unit vector, then taking $p$ to be the limit of $tv,$ as $t$ goes to infinity, gives us:
\[
d(q, p^*) = \lim_{t\rightarrow \infty} \arcsinh\left(\dfrac{1-t q \cdot v}{\sqrt{t^2-1} \sqrt{1-q \cdot q}}\right) = \arcsinh \left(\dfrac{-q \cdot v}{\sqrt{1- q \cdot q}}\right).
\]
\end{example}
\begin{example}
\label{alphaid2}
Suppose we have a point $q$ and a plane $\Pi,$ with the given distance $d(q, \Pi).$ What is the visual half-angle $\alpha_q(\Pi)$ subtended by $\Pi$ at $q?$ To answer this, place $q$ at the origin, and let $\Pi - ( t \mathbf{e}_1)^*.$ The Klein model is conformal at the origin (and only there), so 
\[
\alpha_q(\Pi) = \arccos(1/t).
\]
We need to express $t$ in terms of $d(q, \Pi).$ We know that $d(q, \Pi) = \arcsinh(1/sqrt{t^2-1}), $ so that
\[
t^2 = 1 + \frac{1}{\sinh^2 d(q, \Pi)} = \frac{1}{\tanh^2 d(q, \Pi)},
\]
 so that 
\[
t = \cotanh(d(q, \Pi)),
\]
and finally
\[
\alpha_q(\Pi) = \arccos \tanh d(q, \Pi) = \arcsin \sech d(q, \Pi),
\]
where the last equality follows by a simple computation.
\end{example}
\begin{remark}
The quantity $\alpha_q(\Pi)$ is the classical \emph{angle of parallelism}.
\end{remark}
\begin{example}
What is the radius of the closest point projection of the plane  $\Pi_2$ onto a plane $\Pi_1,$ as a function of the distance $d(\Pi_1, \Pi_2)?$

We place $\Pi_1$ as an equatorial plane with normal pointing along $\mathbf{e}_1,$ and place $\Pi_2$ as a a parallel (in the Klein model) plane, with $\Pi_2^* = p \mathbf{e}_1.$ In this case, the nearest point projection is just orthogonal projection, and its image is an equatorial disk centered at the origin of Euclidean radius $\sqrt{1-1/p^2.} $ It follows that its hyperbolic radius equals
\[
r_{\Pi_1} (\Pi_2) = \arccosh\left(\frac{1}{\sqrt{1-(1-1/p^2)}} \right)= \arccosh(p).
\]
Now we need to find $p$ in terms of $d(\Pi_1, \Pi_2).$ By Example \ref{origdist}, 
\[
d(\Pi_1, \Pi_2) = \arccosh\left( \frac{p}{\sqrt{p^2-1}}\right),
\]
and so
\[
\frac{p^2}{p^2-1} = \cosh^2 d(\Pi_1, \Pi_2),
\]
so that 
\[
p = \cotanh d(\Pi_1, \Pi_2).
\]
So the radius of the projection is
\[
r_{\Pi_1}(\Pi_2) = \arccosh\left(\frac1{\tanh d(\Pi_1, \Pi_2)}\right).
\]
For those who are morally opposed to taking inverse hyperbolic trig functions of hyperbolic trig functions, we can rewrite the expression above. We use
\[
\arccosh x = \log(x + \sqrt{x^2-1}),
\]
where we are taking positive branches of $\arccosh$ and of the square root.
We then write:
\[
\begin{split}
\arccosh(1/\tanh x) &= \\
\log\left(\frac{\cosh x}{\sinh x} + \sqrt{\frac{\cosh^2 x}{\sinh^2 x} - 1}\right) &= \\
\log \left(\frac{\cosh x}{\sinh x} + \sqrt{\frac{\sinh^2 x}{\cosh^2 x}}\right) &= \\
\log \frac{\cosh x + 1}{\sinh x}.
\end{split}
\]
Since \[ \cosh x + 1 = 2 \cosh^2 \frac{x}2,\] while \[\sinh x = 2 \sinh \frac{x}2 \cosh \frac{x}2,\] we get 
\[
\arccosh \frac1{\tanh x} = - \log \tanh \frac{x} 2,
\]
and so
\[
r_{\Pi_1}(\Pi_2) = - \log \tanh \frac{d(\Pi_1, \Pi_2)}2.
\]
\end{example}
\section{Integration}
\label{integration}
In the sequel we will consider pairs of disjoint hyperplanes $\Pi_1, \Pi_2 \subset \mathbb{H}^n.$ It is clear that such a pair can be represented by $p_1^*, p_2^*,$ where $p_1 = t v, \ p_2 = -t v,$ for some arbitrary unit vector $v.$ The hyperplane $\Pi_1$ then passes through $v/t,$ while $\Pi_2$ passes through $-v/t,$ so we know (by formula \eqref{kleindist} that 
\[
d(\Pi_1, \Pi_2) = \arccosh\left(\frac{1+1/t^2}{1-1/t^2}\right) = \arccosh\left(\frac{t^2+1}{t^2-1}\right).
\]
In other words, if we have two hyperplanes at distance $d$ we can arrange them with 
\[
t=\sqrt{\frac{\cosh d + 1}{\cosh d - 1}}.
\]
The convex hull of $\Pi_1$ and $\Pi_2$ as above is given by a cylinder $C(\Pi_1, \Pi_2)$ (Thurston had called this convex hull a \emph{chimney} on occasion). Suppose we want to integrate  function $f$ over $C(\Pi_1, \Pi_2),$ where $f(x) = g(d(x, \Pi_1)).$

To do this we write $x = \alpha \mathbf{v} + \mathbf{w},$ where $\mathbf{v} \cdot \mathbf{w} = 0.$ We can write
\[
d(x, \Pi_1) = \arcsinh \left(\frac{1-t\alpha}{\sqrt{t^2 - 1}\sqrt{1-\alpha^2 - \mathbf{w} \cdot \mathbf{w}}}\right).
\]
The hyperbolic volume element in the Klein model at $x$ is given by
\[
d V_{\mbox{hyp}} = \frac{dx_1 dx_2 \dots d x_n}{(1-x \cdot x)^{\frac{n+1}2}} = \frac{ dx_1 dx_2 \dots dx_n}{(1-\alpha^2 - \mathbf{w} \cdot \mathbf{w})^{\frac{n+1}2}}.
\]
The distance function is rotationally symmetric, so the integral of $g(d(x, \Pi_1))$ over $C(\Pi_1, \Pi_2) $ is best written in cylindrical coordinates thus:
\begin{equation}
\label{chimint}
\begin{split}
I(g, \Pi_1, \Pi_2) = \int_{C(\Pi_1, \Pi_2)} g(d(x, \Pi_1)) d V_{\mbox{hyp}} &= \\
\omega_{n-2} \int_{-1/t}^{1/t} d \alpha \int_0^{\sqrt{1-1/t^2}} \frac{r^{n-2} g(\sqrt{\alpha^2 + r^2})}{(1-\alpha^2-r^2)^{\frac{n+1}2}} d r &= \\
\omega_{n-1}\int_{-\tanh \frac{d(\Pi_1, \Pi_2)}2}^{\tanh \frac{d(\Pi_1, \Pi_2)}2} d\alpha \int_0^{\sech \frac{d(\Pi_1, \Pi_2)}2}\frac{r^{n-2} g(\sqrt{\alpha^2 + r^2})}{(1-\alpha^2-r^2)^{\frac{n+1}2}} d r ,
\end{split}
\end{equation}

where $\omega_k$ is the $k$-dimensional area of the unit sphere $\mathbb{S}^k.$
\begin{remark}
Eq. \eqref{chimint} shows, in particular, that that integral $I(g, \Pi_1, \Pi_2)$ depdends only on $g$ and the distance between $\Pi_1$ and $\Pi_2,$ so we can write it as $I(g, d(\Pi_1, \Pi_2)$ -- of course, this is an immediate consequence of the two-point homogeneity of the space of hyperplanes in $\mathbb{H}^n.$
\end{remark}
\section{Why integrate along chimneys?}
\label{chimneys}
Let $M^n$ be a hyperbolic $n$-manifold with totally geodesic boundary. Consider a point $x \in M^n,$ and  let $\pi_1(x, \partial M)$ be the set of \emph{geodesic} representatives of homotopy classes of paths between $x$ and $\partial M.$ Then:
\begin{lemma}
\begin{equation}
\label{orthoform}
\sum_{\gamma \in \pi_1(x, \partial M)} V_{n-1}(\arcsin \sech \ell(\gamma)) = \omega_{n-1}.
\end{equation}
\end{lemma}
\begin{proof} Let $\widetilde{M}^n$ be the covering space of $M^n$ corresponding to the boundary. Then, the preimage under the covering map of $\partial M$ is a collection of hyperplanes which, together, cover all of $\partial \mathbb{H}^n,$ except for a set of measure $0.$ Picking some lift $\widetilde{x}$ of $x$ the result follows by applying the basic result Eq. \eqref{asineq}.
\end{proof}

Integrating Eq. \eqref{orthoform} over $M,$ we obtain
\begin{equation}
\label{volform}
\int_{x\in M^n} \sum_{\gamma \in \pi_1(x, \partial M)} V_{n-1}(\arcsin \sech \ell(\gamma)) = \omega_{n-1} V(M^n).
\end{equation}
Equation \eqref{volform} seems not very useful, since the left hand side looks a little complicated, but look again at the covering space $\widetilde{M}^n$ corresponding to $\partial M^n.$ Over each component $C$ of the lift of $\partial M^n$ we have a "big chimney' (geometrically a halfspace), which (since almost all of $\partial \mathbb{H}^n$ is covered by lift of $\partial M^n$) splits into chimneys corresponding to the other components of the lift of $\partial M^n.$ This means that Eq. \eqref{volform} can be written as 
\[
\sum_{\mbox{$\gamma$ in orthospectrum of $M^n$}} I(V_{n-1} \circ \arcsin  \circ\sech, \ell(\gamma)) = \omega_{n-1} V(M).
\]
\section{Not hyperplanes}
\label{notchimneys}
The idea of taking shadows on the visual sphere of a point can be extended to the case when we don't have any boundary. For a nice application, recall that the \emph{Poincar\'e series} associated to a group $\Gamma$ of isometries of $\mathbb{H}^n$ is defined as 
\begin{equation}
\label{poincare}
P(x, y, s, \Gamma) = \sum_{\gamma \in \Gamma} \exp(-s d(x, \gamma(y))).
\end{equation}
Furthermore, we define \emph{the exponent of convergence} of $\Gamma$ as 
\[
\delta(\Gamma) = \inf\{s \left| P(x, y, s, \Gamma) < \infty\right.\}.
\]
\emph{A priori} the above expression depends on $x$ and $y,$ but it is not hard to see that its convergence does not.

Finally, the group $\Gamma$ is called \emph{a group of divergence type} if $P(x, y, \delta(\Gamma), \Gamma) = \infty,$ and it is called \emph{a group of convergence type} otherwise.
It is known that geometrically finite groups of isometries of $\mathbb{H}^n$ are of divergence type, but here we give a particularly easy proof of that fact for groups acting co-compactly on $\mathbb{H}^n.$
Indeed, suppose that $\Gamma$ acts co-compactly. Let $P$ be a (pre)compact fundamental domain of $\Gamma.$ We know that for any $x\in P,$ there is a $d_x$ such that $P$ is contained in a ball $B(x, d_x)$ of radius $d_x$ around $x.$ Now, consider $B(\gamma(x), d_x).$ What is its visual diameter from some fixed basepoint $y?$ The following can be obtained by a straightoforward computation:
\begin{lemma}
If $d(\gamma(x), y) \gg 1,$ then the visual diameter of $B(\gamma(x), d_x)$ as viewed from $y$ is asymptotic to $4 \tanh d_x exp(-d(\gamma(x), y)).$
\end{lemma}
It follows that the spherical volume of the projection of $B(\gamma(x), d_x)$ onto the visual sphere of $y$ is asymptotic to $ C \exp(-(n-1) d(\gamma(x), y)),$ and since the images of the fundamental domains cover the entire horizon at $y,$ so do the images of the containing balls, so $\sum_{\gamma \in \Gamma} \exp(-(n-1) d(\gamma(x), y) > c,$ for some constant $c.$
Notice, however, now that if we let $\Gamma^\prime = \Gamma \backslash \{\gamma_1, \gamma_2, \dotsc, \gamma_k\}$ for \emph{any} finite subset $\{\gamma_1, \dots, \gamma_k\} \subset \Gamma,$ then the same hold: any ray shot out of $y$ intersects infinitely many images of the fundamental domain,
so that $\sum_{\gamma \in \Gamma^\prime} \exp(-(n-1) d(\gamma(x), y) > c,$ for any such subset, from which it follows that the Poincar\'e series of $\gamma$ diverges for $s=n-1,$ which is, in fact, the exponent of convergence for a co-compact group of isometries of $\mathbb{H}^n.$ The above argument in reality gives us much more, it shows  that 
\[
\sum_{\gamma \in \Gamma; d(\gamma(x), y) < R} \exp(-(n-1) d(\gamma(x), y)) \sim R,
\]
for any co-compact $\Gamma.$

These results are close in spirit to the beautiful (and beautifully written) work of Dennis Sullivan \cite{dennisIHES,dennisdense}, which we recommend wholeheartedly to the reader.

\bibliographystyle{plain}
\bibliography{basma}

\begin{thebibliography}{10}

\bibitem{andreev1}
E.~M. Andreev.
\newblock Convex polyhedra in {L}oba\v cevski\u\i\ spaces.
\newblock {\em Mat. Sb. (N.S.)}, 81 (123):445--478, 1970.

\bibitem{andreev2}
E.~M. Andreev.
\newblock Convex polyhedra of finite volume in {L}oba\v cevski\u\i\ space.
\newblock {\em Mat. Sb. (N.S.)}, 83 (125):256--260, 1970.

\bibitem{andreevsharp}
E.~M. Andreev.
\newblock The intersection of the planes of the faces of polyhedra with sharp
  angles.
\newblock {\em Mat. Zametki}, 8:521--527, 1970.

\bibitem{basmajianortho}
Ara Basmajian.
\newblock The orthogonal spectrum of a hyperbolic manifold.
\newblock {\em Amer. J. Math.}, 115(5):1139--1159, 1993.

\bibitem{bridgekahn}
Martin Bridgeman and Jeremy Kahn.
\newblock Hyperbolic volume of manifolds with geodesic boundary and
  orthospectra.
\newblock {\em Geom. Funct. Anal.}, 20(5):1210--1230, 2010.

\bibitem{calespots}
Danny Calegari.
\newblock Chimneys, leopard spots and the identities of {B}asmajian and
  {B}ridgeman.
\newblock {\em Algebr. Geom. Topol.}, 10(3):1857--1863, 2010.

\bibitem{coxref}
H.~S.~M. Coxeter.
\newblock Discrete groups generated by reflections.
\newblock {\em Ann. of Math. (2)}, 35(3):588--621, 1934.

\bibitem{mcshaneid}
Greg McShane.
\newblock Simple geodesics and a series constant over {T}eichmuller space.
\newblock {\em Invent. Math.}, 132(3):607--632, 1998.

\bibitem{potyavinberg}
Leonid Potyagailo and Ernest Vinberg.
\newblock On right-angled reflection groups in hyperbolic spaces.
\newblock {\em Comment. Math. Helv.}, 80(1):63--73, 2005.

\bibitem{riasymp}
Igor Rivin.
\newblock Asymptotics of convex sets in {E}uclidean and hyperbolic spaces.
\newblock {\em Adv. Math.}, 220(4):1297--1315, 2009.

\bibitem{rioneint}
Igor Rivin.
\newblock Geodesics with one self-intersection, and other stories.
\newblock {\em Adv. Math.}, 231(5):2391--2412, 2012.

\bibitem{dennisIHES}
Dennis Sullivan.
\newblock The density at infinity of a discrete group of hyperbolic motions.
\newblock {\em Inst. Hautes \'Etudes Sci. Publ. Math.}, (50):171--202, 1979.

\bibitem{dennisdense}
Dennis Sullivan.
\newblock On the ergodic theory at infinity of an arbitrary discrete group of
  hyperbolic motions.
\newblock In {\em Riemann surfaces and related topics: {P}roceedings of the
  1978 {S}tony {B}rook {C}onference ({S}tate {U}niv. {N}ew {Y}ork, {S}tony
  {B}rook, {N}.{Y}., 1978)}, volume~97 of {\em Ann. of Math. Stud.}, pages
  465--496. Princeton Univ. Press, Princeton, N.J., 1981.

\bibitem{vinbergref}
{\`E}.~B. Vinberg.
\newblock Hyperbolic groups of reflections.
\newblock {\em Uspekhi Mat. Nauk}, 40(1(241)):29--66, 255, 1985.

\bibitem{vinbergsh}
{\`E}.~B. Vinberg and O.~V. Shvartsman.
\newblock Discrete groups of motions of spaces of constant curvature.
\newblock In {\em Geometry, {II}}, volume~29 of {\em Encyclopaedia Math. Sci.},
  pages 139--248. Springer, Berlin, 1993.

\end{thebibliography}
\end{document}